\documentclass[a4paper,11pt,onecolumn,oneside]{article}
\usepackage{calc,amsmath,amssymb,amsfonts,amsthm}
\usepackage{graphicx}
\usepackage[linesnumbered, boxed]{algorithm2e}
\usepackage{bm}
\usepackage{tikz}
\usepackage[]{todonotes}   
\usepackage{url}
\usepackage{multicol}
\usepackage{setspace}

\newcommand{\spc}{\vspace{0.5cm}}

\newcommand{\A}{{\bf{A}}}

\makeatletter
\providecommand*{\cupdot}{%
  \mathbin{%
    \mathpalette\@cupdot{}%
  }%
}
\newcommand*{\@cupdot}[2]{%
  \ooalign{%
    $\m@th#1\cup$\cr
    \hidewidth$\m@th#1\cdot$\hidewidth
  }%
}
\makeatother

\allowdisplaybreaks 

\newcommand{\footremember}[2]{%
    \footnote{#2}
    \newcounter{#1}
    \setcounter{#1}{\value{footnote}}%
}
\newcommand{\footrecall}[1]{%
    \footnotemark[\value{#1}]%
} 

\theoremstyle{plain} \newtheorem{lemmax}{\textbf{Lemma}}
\theoremstyle{remark} 
\theoremstyle{plain} \newtheorem{theoremx}{\textbf{Theorem}}
\theoremstyle{plain} 
\theoremstyle{plain} 
\theoremstyle{definition} 
\theoremstyle{plain} 
\theoremstyle{plain} 

\title{Clustering via Hypergraph Modularity}
\author{%
  Bogumi\l{} Kami\'nski\footremember{bkaff}{SGH Warsaw School of Economics, Warsaw, Poland}
  \and Val\'{e}rie Poulin\footremember{vpaff}{The Tutte Institute for Mathematics and Computing, Ottawa, ON, Canada.}%
  \and Pawe\l{} Pra\l{}at\footremember{ppaff}{Department of Mathematics, Ryerson University, Toronto, ON, Canada.}%
  \and Przemys\l{}aw Szufel\footrecall{bkaff}%
  \and Fran\c{c}ois Th\'{e}berge\footrecall{vpaff}%
}
\begin{document}

\maketitle

\begin{abstract}
Despite the fact that many important problems (including clustering) can be described using hypergraphs, theoretical foundations as well as practical algorithms using hypergraphs are not well developed yet. In this paper, we propose a hypergraph modularity function that generalizes its well established and widely used graph counterpart measure of how clustered a network is.
In order to define it properly, we generalize the Chung-Lu model for graphs to hypergraphs.
We then provide the theoretical foundations to search for an optimal solution with respect to our hypergraph modularity function. Two simple heuristic algorithms are described and applied to a few small illustrative examples.
We show that using a strict version of our proposed modularity function often leads to a solution where a smaller number of hyperedges get cut as compared to optimizing modularity of 2-section graph of a hypergraph.
\end{abstract}

{\bf Keywords}: graph theory; hypergraph; modularity; clustering; community detection

\section{Introduction}
\label{section:intro}

An important property of complex networks is their community structure, that is, the organization of vertices in clusters, with many edges joining vertices of the same cluster and comparatively few edges joining vertices of different clusters~\cite{Fortunato,Girvan}. In social networks communities may represent groups by interest, in citation networks they correspond to related papers, in the web communities are formed by pages on related topics, etc. Being able to identify communities in a network could help us to exploit this network more effectively. For example, clusters in citation graphs may help to find similar scientific papers, discovering users with similar interests is important for targeted advertisement, clustering can also be used for network compression and visualization.

\spc

The key ingredient for many clustering algorithms is \textit{modularity}, which is at the same time a global criterion to define communities, a quality function of community detection algorithms, and a way to measure the presence of community structure in a network. Modularity for graphs was introduced by Newman and Girvan~\cite{Mod2} and it is based on the comparison between the actual density of edges inside a community and the density one would expect to have if the vertices of the graph were attached at random regardless of community structure, while respecting the vertices' degree on average. This random family of graphs is known as the Chung-Lu random model \cite{chunglu}. 

\spc 

Myriad of problems can be described in hypergraph terms, however, despite being formally defined in the 1960s (and various realizations studied long before that) hypergraph theory is patchy and often not sufficiently general.  The result is a lack of machinery for investigating hypergraphs, leading researchers and practitioners to create the 2-section graph of a hypergraph of interest \cite{Zhou2007, Shi2000, Rodriguez2003, Zien1999, Agarwal2005, Agrawal2006} or to restrict their study to $d$-uniform hypergraphs \cite{Karypis2007, chien2018}. In taking the 2-section (that is, making each hyperedge a clique) we lose some information about edges of size greater than two. Sometimes losing this information does not affect our ability to answer questions of interest, but in other cases it has a profound impact. In particular, an important scenario when hypergraph-based approach can be preferred, is when a large hyperedge connecting some vertices is a strong indicator that they all belong to the same community. Such situations occur often in practice. Let us briefly discuss two simple examples. First consider e-mails as hyperedges of a hypergraph whose vertices are e-mail addresses. Multiple addresses in an e-mail (group e-mails) most likely are not sent to random people, but rather to some community of common interests. As a second example consider a platform like GitHub, where vertices are users and hyperedges are repositories linking users that committed to them. Again, if a group of users commits to the same repository it is most likely a strong indicator that they form some community. In both cases, as indicated above, replacing a hyperedge by a clique would lose valuable information. 

\spc

The paper is organized as follows.
In Section \ref{section:qH}, we review the Chung-Lu model for graphs and its link
to the modularity function. We then propose a generalization of the Chung-Lu model 
for hypergraphs, as well as a hypergraph modularity function.
In Section \ref{section:search}, we provide the framework to develop algorithms 
using our hypergraph modularity function.
Two simple heuristic algorithms are described in Section \ref{section:algos}, 
and applied to a few small illustrative examples in Section \ref{section:ex}.
This is a new measure we are proposing, and there is plenty of future work
to do, which we summarize in Section \ref{section:conclusion}.




\section{Hypergraph Modularity}
\label{section:qH}

In this section, we recall the definition of modularity function for graphs, and we propose its generalization for hypergraphs.
Throughout the paper we will use $n$ for the number of vertices. 
We will use $\binom{X}{k}$ for the set consisting of all $k$-element subsets of $X$. Finally, $[n]:=\{1, \ldots, n\}$.

\subsection{Chung-Lu Model for Graphs}

Let $G=(V,E)$ be a graph, where $V = \{v_1,\ldots ,v_n\}$ are the vertices, the edges $E$ are multisets of $V$ of cardinality 2 (loops allowed), and $deg_G(v)$ is the degree of $v$ in $G$ (with a loop at $v$ contributing 2 to the degree of $v$).  For $A \subseteq V$, let the {\it volume} of $A$ be $vol_G(A) = \sum_{v\in A} deg_G(v)$; in particular $vol_G(V) = \sum_{v\in V} deg_G(v) = 2|E|$. 
We will omit the subscript $G$ when the context is clear.

\spc

We define $\mathcal{G}(G)$ to be the probability distribution on graphs on the vertex set $V$ following the well-known Chung-Lu model~\cite{CL, Seshadhri2012, Kolda2014, Winlaw2015}. In this model, each set $e=\{v_i,v_j\}$, $v_i,v_j \in V$, is independently 
sampled as an edge with  probability given by:

\[
P(v_i,v_j) = 
\begin{cases}
\frac{deg(v_i) deg(v_j)}{2|E|}, & i \ne j \\
\frac{deg^2(v_i)}{4|E|}, & i = j.
\end{cases}
\]

Note that this model allows for loops (even if $G$ itself is a simple graph). Clearly, the model depends of the choice of $G$ but, in fact, it is only a function of the degree sequence of $G$. One desired property of this random model is that it yields a distribution that preserves the expected degree for each vertex, namely: for any $i \in [n]$,

\begin{eqnarray*}
\mathbb{E}_{G' \sim \mathcal{G}(G)}[deg_{G'}(v_i)] &=& \sum_{j \in [n] \setminus \{i\}} \frac{deg(v_i)deg(v_j)}{2|E|} + 2 \cdot \frac{deg^2(v_i)}{4|E|} \\
&=& \frac{deg(v_i)}{2|E|} \sum_{j \in [n]} deg(v_j) ~~=~~ deg(v_i),
\end{eqnarray*}

where all degrees are with respect to graph $G$.
This model will be useful to understand the graph modularity definition and its generalization to hypergraphs.

\subsection{Review of Graph Modularity}

The definition of modularity for graphs was first introduced by Newman and Girvan in~\cite{Mod2}. 
Despite some known issues with this function such as the ``resolution limit'' reported in \cite{Fortunato2007},
many popular algorithms for partitioning large graph data sets use it~\cite{CNM,Mod1,Newman}. 
It was also recently studied for some models of complex networks~\cite{random_graphs1,random_graphs2,random_graphs3}.
The modularity function favours partitions in which a large proportion of the edges fall entirely within the parts and biases against having too few or too unequally sized parts. 

\spc

For a graph $G=(V,E)$ and a given partition $\A = \{A_1, \ldots, A_k\}$ of $V$, the modularity function is defined as follows:

\begin{equation}\label{eq:q_G_A}
q_G(\A) = \frac{1}{|E|}\sum_{A_i \in \A} \big( e_G(A_i) -  \mathbb{E}_{G' \sim \mathcal{G}(G)}[ e_{G'}(A_i)] \big),
\end{equation}

where $e_G(A_i) = |\{ \{v_j,v_k\} \in E : v_j, v_k \in A_i\}|$ is the number of edges in the subgraph of $G$ induced by the set $A_i$. The modularity measures the deviation of the number of edges of $G$ that lie inside parts of $\A$ from the corresponding expected value based on the Chung-Lu distribution $\mathcal{G}(G)$. The expected value for part $A_i$ is

\begin{eqnarray*}
\mathbb{E}_{G' \sim \mathcal{G}(G)}[ e_{G'}(A_i)] &=&
\sum_{\{v_j,v_k\} \in \binom{A_i}{2}} \frac{deg(v_j)deg(v_k)}{2|E|} + \sum_{v_j \in A_i} \frac{deg^2(v_j)}{4|E|} \\
&=& \frac{1}{4|E|} \left( \sum_{v_j \in A_i}deg(v_j)\right)^2 = \frac{(vol(A_i))^2}{4|E|}. 
\end{eqnarray*}

The first term in~(\ref{eq:q_G_A}), $\sum_{A_i \in \A} e_G(A_i)/|E|$, is called the \emph{edge contribution}, whereas the second one, $\sum_{A_i \in \A} (vol(A_i))^2 / (4|E|)$, is called the \emph{degree tax}. 
It is easy to see that $q_G(\A) \le 1$. 
Also, if $\A = \{V\}$, then $q_G(\A) = 0$, and
if $\A = \{ \{v_1\}, \ldots, \{v_n\}\}$, then 
$q_G(\A) = - \frac {\sum deg(v)^2}{4|E|^2} < 0.
$ This is often used as a starting point for algorithms, including the ones we investigate in this paper. 

\spc

The \emph{modularity} $q^*(G)$ is defined as the maximum of $q_G(\A)$ over all possible partitions $\A$ of $V$; that is,
$
q^*(G) = \max_{\A} q_G(\A).
$
In order to maximize $q_G(\A)$ one wants to find a partition with large edge contribution subject to small degree tax. If $q^*(G)$ approaches 1 (which is the trivial upper bound), we observe a strong community structure; conversely, if $q^*(G)$ is close to zero (which is the trivial lower bound), there is no community structure. The definition in (\ref{eq:q_G_A}) can be generalized to weighted edges by replacing edge counts with sums of edge weights.

\subsection{Generalization of Chung-Lu Model to Hypergraphs}

Consider a hypergraph $H = (V, E)$ with $V=\{v_1, \ldots, v_n\}$, where hyperedges $e \in E$ are subsets of $V$ of cardinality greater than one. Since we are concerned with not necessarily simple hypergraphs, hyperedges are multisets. Such hyperedges can be described using distincts sets of pairs $e = \{(v,m_e(v)): v \in V\}$ where $m_e(v) \in \mathbb{N} \cup \{0\}$ is the multiplicity of the vertex $v$ in $e$ (including zero which indicates that $v$ is not present in $e$). Then $|e| = \sum_v m_e(v)$ is the {\it size} of hyperedge $e$ and the {\it degree of a vertex} $v$ in $H$ is defined as $deg_H(v) = \sum_{e \in E} m_e(v)$. When the reference to the hyperedge is clear from the context, we simply use $m_i$ to denote $m_e(v_i)$.

\spc

A hypergraph is said to be $d$-{\it uniform} if all its hyperedges have size $d$. In particular, $2$-uniform hypergraph is simply a graph. All hypergraphs $H$ can be expressed as the disjoint union of $d$-uniform hypergraphs
$H = \bigcup H_d$, where $H_d = (V, E_d)$, $E_d \subseteq E$ are all hyperedges of size $d$, and $deg_{H_d}(v)$ is the $d$-degree of vertex $v$. 
Just as for graphs, the volume of a vertex subset $A \subseteq V$ is $vol_H(A) = \sum_{v \in A} deg(v)$.


\spc

Similarly to what we did for graphs, we define a random model on hypergraphs, $\mathcal{H}(H)$, where the expected degrees of all vertices are the degrees in $H$.
To simplify the notation, we omit the explicit references to $H$ in the remaining of this section; in particular, $deg(v)$ denotes $deg_H(v)$, $\mathcal{H}$ denotes $\mathcal{H}(H)$, $E_d$ denotes the edges of $H$ of size $d$, etc. Moreover, we use $E'$ to denote the edge set of $H'$.

\spc

Let $F_d$ be the family of multisets of size $d$; that is, 

$$
F_d := \left\{ \{(v_i,m_i): 1 \le i \le n\} : \sum_{i=1}^n m_i = d \right\}.
$$

The hypergraphs in the random model 
are generated via independent random experiments. 
For each $d$ such that $|E_d|>0$, the probability of generating 
the edge $e \in F_d$ is given by:

\begin{eqnarray}
P_{\mathcal{H}}(e) = |E_d| \cdot {d \choose m_1,\ldots,m_n} \prod_{i=1}^n \left(
\frac{deg(v_i)}{vol(V)}\right)^{m_i}.
\label{eq:prob_tot} 
\end{eqnarray}

Let $(X_1^{(d)}, \ldots, X_n^{(d)})$ be the random vector 
following a multinomial distribution 
with parameters $d, p_{\mathcal{H}}(1), \ldots, p_{\mathcal{H}}(n)$; that is,

$$
s_{\mathcal{H}}(e) := \mathbb{P} \Big( (X_1^{(d)}, \ldots, X_n^{(d)}) = (m_1, \ldots m_n) \Big) = {d \choose m_1,\ldots,m_n} \prod_{i=1}^n (p_{\mathcal{H}}(i))^{m_i}
$$

where $p_{\mathcal{H}}(i) = deg(v_i)/vol(V)$ and $\sum_{i \in [n]} p_{\mathcal{H}}(i) = 1$. 
Note that this is the expression found in (\ref{eq:prob_tot}); 
that is, $P_{\mathcal{H}}(e) = |E_d| \cdot s_{\mathcal{H}}(e)$.
An immediate consequence is that the expected number of edges of size $d$ is $|E_d|$. 
Finally, as with the graph Chung-Lu model, we assume that all $P_{\mathcal{H}}(e) \le 1$.

\spc

In order to compute the expected $d$-degree of a vertex $v_i \in V$, 
note that 

$$
deg_{H'_d}(v_i) = \sum_{e \in F_d} m_e(v_i) \cdot \mathbb{I}_{\{e \in E'\}},
$$

where $\mathbb{I}_{\{\}}$ is the indicator random variable. 
Hence, using the linearity of expectation, then splitting the sum into $d+1$ partial sums for different multiplicities of $v_i$, we get:

\begin{align*}
\mathbb{E}_{H' \sim \mathcal{H}} \big( deg_{H'_d}(v_i) \big) &= \sum_{e \in F_d} m_e(v_i) \cdot P_{\mathcal{H}}(e) = |E_d| \sum_{e \in F_d} m_e(v_i) \cdot s_{\mathcal{H}}(e) \\
& = |E_d| \sum_{m=0}^d m \sum_{e \in F_d; m_e(v_i)=m} s_{\mathcal{H}}(e) \\
& = |E_d| \sum_{m=0}^d m \cdot \mathbb{P}(X_i^{(d)} = m) \\
& = |E_d| \sum_{m=0}^d m \cdot {d \choose m} (p_{\mathcal{H}}(i))^m (1-p_{\mathcal{H}}(i))^{d-m} \\
& = |E_d| \cdot d \cdot p_{\mathcal{H}}(i).
\end{align*}

The second last equality follows from the fact that we obtained the expected value of a random variable with binomial distribution.
One can compute the expected degree as follows:

$$
\mathbb{E}_{H' \sim \mathcal{H}}[deg_{H'}(v_i)] =
\sum_{d\ge 2} \frac{d \cdot |E_d| \cdot deg(v_i)}{vol(V)} = deg(v_i),
$$
since $vol(V) = \sum_{d\ge 2} d\cdot |E_d|$. 

\spc

We will use the generalization of the Chung-Lu model to hypergraphs as a null model allowing us to define hypergraph modularity.

\subsection{Strict Hypergraph Modularity}

Consider a hypergraph $H=(V, E)$ and $\A = \{A_1, \ldots, A_k\}$, a partition of $V$.
For edges of size greater than 2, several definitions can be used to quantify the edge contribution given
$\A$, such as:

\begin{itemize}
\item [(a)] all vertices of an edge have to belong to one of the parts to contribute; 
this is a \emph{strict} definition that we focus on in this paper; 
\item [(b)] the majority of vertices of an edge belong to one of the parts;
\item [(c)] at least 2 vertices of an edge belong to the same part; this is implicitly used when we replace a hypergraph with its 2-section graph representation.
\end{itemize}

We see that the choice of hypergraph modularity function is not unique; in fact, it depends on how strongly we believe that an existence of a hyperedge is an indicator that vertices belonging to it fall into one community. More importantly, one needs to decide how often vertices in one community ``blend'' together with vertices from other community; that is, how hermetic the community is. In particular, option (c) is the \emph{softest} one and leads to standard 2-section graph modularity. Therefore, following the motivation presented in the Introduction in this text, we concentrate on the second extreme, option (a), that we call \emph{strict}. In this case, the definition of edge contribution for $A_i \subseteq V$ is:

\begin{equation}\label{eq:ec}
e_H(A_i) = |\{ e \in E ;~ e \subseteq A_i \}|.
\end{equation}

Again, this is a natural choice: for 2-section, an edge contributes as long as at least 2 vertices belong to the same part (the weakest condition) whereas for the strict modularity, the requirement is that \emph{all} vertices belong to the same part (the strongest condition). Analyzing it will allow us to investigate if differentiation between strict hypergraph modularity and 2-section graph modularity leads to qualitatively different results. All other possible definitions would cover the space between the two extremes we consider. We concentrate on strict modularity but we will show how to easily generalize our formulas to other scenarios. 

\spc

The strict modularity of $\A$ on $H$ is then defined as a natural extension of standard modularity in the following way:

\begin{equation}\label{eq:q_A_H}
q_H(\A) = \frac{1}{|E|} \sum_{A_i \in \A} \left( e_H(A_i) - \mathbb{E}_{H' \sim \mathcal{H}}[e_{H'}(A_i)]  \right).
\end{equation}

\subsubsection{A Formula for Expected Edge Contribution}

Consider any $A \subseteq V$. We want to compute the expected edge contribution of $A$ over $\mathcal{H}$. 
Let $F_d(A) \subseteq F_d$ be the family of multisets of size $d$ with all members in $A$; that is, 

$$
F_d(A) := \left\{ \{(v_i,m_i): 1 \le i \le n\} : \sum_{i=1}^n m_i = \sum_{i:v_i \in A} m_i = d \right\}.
$$

First, note that 

\begin{eqnarray*}
\mathbb{E}_{H' \sim \mathcal{H}}[e_{H'}(A)] & = & \sum_{d \geq 2} \sum_{e \in F_d(A)} P_{\mathcal{H}}(e) = \sum_{d \geq 2} |E_d| \sum_{e \in F_d(A)} s_{\mathcal{H}}(e) \\
 &=& \sum_{d \geq 2} |E_d| \cdot \mathbb{P}\left( \sum_{i; v_i \in A} X_i^{(d)} = d \right) 
 = \sum_{d \geq 2} |E_d| \cdot (p_{\mathcal{H}}(A))^d
\end{eqnarray*}

where $p_{\mathcal{H}}(A) = \sum_{i; v_i\in A} p_{\mathcal{H}}(i)$,
therefore $p_{\mathcal{H}}(A) = vol(A)/vol(V)$, so

\begin{eqnarray}
\mathbb{E}_{H' \sim \mathcal{H}}[e_{H'}(A)] & =
\sum_{d \geq 2} |E_d| \cdot (vol(A)/vol(V))^d.
\label{eq:exp_contrib}
\end{eqnarray}

Putting (\ref{eq:exp_contrib}) properly into equation~(\ref{eq:q_A_H}), we get the strict modularity function of a hypergraph partition:

\begin{eqnarray} \label{eq:modularity}
q_{H}({\mathbf A}) &=& \frac{1}{|E|}  \left(\sum_{A_i \in {\bf A}} e(A_i)   - 
\sum_{d \geq 2} |E_d|\sum_{A_i \in {\bf A}} \left( \frac{vol(A_i)}{vol(V)} \right)^d \right).
\label{eq:q_H}
\end{eqnarray}

Just as for graphs, the corresponding \emph{modularity} $q_{H}^*$ is 
defined as the maximum of $q_{H}(\A)$ over all possible partitions $\A$ of $V$.

\subsubsection{Independent modularities}\label{sec:independent}

In some applications, we may want to consider the modularity independently over various
subsets of edges. For examples, a hypergraph may consist of hyperedges obtained from 
different sources (for example, perhaps we collect data from several independent social networks), or we may want to handle hyperedges of different sizes separately.

\spc

Let $H=(V,E)$ and write $E = \bigcup_{i=1}^k E_i$, a disjoint union of the hyperedges.
We let $H_i=(V,E_i)$ and we define a independent modularity function:

$$
q_H^I({\mathbf A}) = \sum_{i=1}^k  w_i q_{H_i}({\mathbf A})
$$

with the $q_{H_i}({\mathbf A})$ as defined in (\ref{eq:q_H}), and $w_i$ some weights such that $\sum_{i=1}^k w_i = 1$,
with natural choice being $w_i = |E_i| / |E|$.

\spc

For example, if we decompose $H$ into $d$-uniform hypergraphs $H_d$, we
get the following degree-independent modularity function using the 
natural weights:

$$
q_H^{DI}({\mathbf A}) = \sum_{d \ge 2} \frac{|E_d|}{|E|} q_{H_d}({\mathbf A}),
$$

where $H_d=(V,E_d)$ contains all edges of size $d$ in $H$.
This corresponds to (\ref{eq:q_H}) replacing the volumes computed over $H$
with volumes computed over $H_d$ for each $d$ where $|E_d|>0$.

\subsubsection{Generalizations}\label{sec:generalization}

As with graphs, one can easily generalize the modularity function to allow for weighted hyperedges. Let us also mention that we focus on the strict definition in this paper but
it is straightforward to adjust the degree tax to many natural definitions of edge contribution. In particular, for the majority definition (see (b) at the beginning of this section), one can simply replace 
$\mathbb{P} \left( \sum_{i; v_i \in A} X_i^{(d)} = d \right)$ 
with 
$\mathbb{P} \left( \sum_{i; v_i \in A} X_i^{(d)} > d/2 \right)$ 
in (\ref{eq:exp_contrib}).

\section{Searching the Solution Space}
\label{section:search}

In this section, we show that the solution that maximizes (\ref{eq:q_H}) lies in a subset of ${\cal P}(V)$ of size at  most $2^{|E|}$ avoiding the search of the full set ${\cal P}(V)$.
This will be useful in designing efficient heuristic algorithms.

\spc

Let ${\cal S}(H)$ denote the set of all sub-hypergraphs of $H=(V,E)$ on the vertex set $V$: 
${\cal S}(H) = \{H'=(V,E')~|~ E' \subseteq E\}$. We use $|H'|$ to denote $|E'|$, the number of edges in $H'$. 
Moreover, let $p: {\cal S}(H) \rightarrow {\cal P}(V)$ denote the function that sends a sub-hypergraph of $H$ to the partition its connected components induce on $V$. 
We define a relation on ${\cal S}(H)$:
 $$H_1 \equiv_p H_2 \iff p(H_1)=p(H_2)$$
that puts two sub-hypergraphs in relation if they have identical connected components.
Since $\equiv_p$ is an equivalence relation (based on equality), we can define the quotient set ${\cal S}(H)/_{\equiv_p}$. 
This quotient set contains equivalence classes that are in bijection 
with the set of all different vertex partitions that can be induced by the union of elements of $E$. 
Its cardinality depends on $E$ but is at most $2^{|E|}$; however, it is typically much smaller than this trivial upper bound.

\spc

Now, let us define the {\it canonical representative mapping} which identifies a natural representative member for each equivalence class. The canonical representative mapping
$f : {\cal S}(H)/_{\equiv_p} \rightarrow {\cal S}(H)$ maps an equivalence class to the largest member of this class:
$f([H']) = H^*$ where  $H^* \in [H']$ and $|H^*| \geq |H''|$ for all $H'' \in [H'].$
This function is well-defined; indeed, if $H_1, H_2 \in [H']$, then the union of $H_1$ and $H_2$ is also in $[H']$ and so it is impossible that two members have the largest size.
Its outcome is the subgraph $H^* = (V, E^*)$ whose edge set is the union of edges of all members of the equivalence class. The following lemma explains 
why the canonical representative is {\it natural} with respect to the definition of strict modularity.
As this observation follows easily from definitions, the proof is omitted. 


\begin{lemmax}
Let $H = (V,E)$ be a hypergraph and ${\bf A}=\{A_1,\ldots,A_k\}$ be any partition of $V$. If there exists $H' \in {\cal S}(H)$ such that ${\bf A} = p(H')$, then the
edge contribution of the strict modularity of ${\bf A}$ is $\frac{|E^*|}{|E|}$, where $E^*$ is the edge set of the canonical representative of $[H']$.
\end{lemmax}

\spc

The set of canonical representatives, the image of $f$, is a subset of ${\cal S}(H)$. We denote this set by ${\cal S^*}(H)$ and the image of $p$ restricted to ${\cal S^*}(H)$ by ${\cal P^*}(V)$.

The next Lemma shows how the degree tax behaves on partition refinement.
\begin{lemmax}
Let $H = (V,E)$ be a hypergraph and ${\bf A}$ be any partition of $V$. If ${\bf B}$ is a refinement of ${\bf A}$, then the degree tax of ${\bf A}$ is larger than or equal to the degree tax of ${\bf B}$ 
and it is equal if and only if ${\bf A} = {\bf B}$.
\end{lemmax}

\begin{proof} Let ${\bf A}=\{A_1,\ldots,A_k\}$. Since ${\bf B}$ is a refinement of ${\bf A}$, for each part of ${\bf A}$, $A_i$, there exists ${\bf B}_i$, a subset of parts of ${\bf B}$, such that
$A_i = \bigcup_{B \in {\bf B}_i}B$ and ${\bf B} = \bigcup_i {\bf B}_i$. Hence, for each $A_i$ and for each $d$, we have that $vol_d(A_i) = \sum_{B \in {\bf B}_i} vol_d(B)$ and so
$$vol_d(A_i)^d = \left(\sum_{B \in {\bf B}_i} vol_d(B) \right)^d \geq \sum_{B \in {\bf B}_i} vol_d(B)^d.$$
The equality holds if and only if $|{\bf B}_i| = 1$ for all $i$. The result follows.
\end{proof}

\spc

The next result, the main result of this section, shows that one can restrict the search space to canonical representatives from ${\cal S^*}(H)$.

\begin{theoremx}
Let $H = (V,E)$ be a hypergraph. If ${\bf A} \in {\cal P}(V)$ maximizes the strict modularity function $q_H(\cdot)$, then 
${\bf A} \in {\cal P}^*(V)$.
\end{theoremx}

\begin{proof}
Assume that ${\bf A}=\{A_1,\ldots,A_k\}$ maximizes the strict modularity function $q_H(\cdot)$. We will show that there exists $H^*=(V,E^*) \in {\cal S}^*(H)$ such that $q_H(p(H^*)) \geq q_H({\bf A})$.
Let $E^* = \{e \in E : e \subseteq A_i \text{ for some } i \}$. 
By construction of $H^*$, the (strict) edge contribution of partitions ${\bf A}$ and $p(H^*)$ are identical. 
Again, from construction, note that the partition $p(H^*)$ is a refinement of ${\bf A}$. Hence, the previous Lemma states that the degree tax of ${\bf A}$ is larger than or 
equal to the degree tax of $p(H^*)$. With equal edge contribution, this means that $q_H(p(H^*)) \geq q_H({\bf A})$. Since ${\bf A}$ is an optimal solution, the equality must 
hold which is only possible if ${\bf A} = p(H^*)$.
\end{proof}

\subsection{Illustration}
\label{section:illustration}

We illustrate the concepts of equivalence classes and vertex partitions with a simple hypergraph $H=(V,E)$
shown in Figure \ref{fig1}, where $V=\{v_1,v_2,v_3,v_4,v_5\}$ and $E=\{e_1,e_2,e_3\}$, with $e_1=\{v_1,v_2,v_3\}$,
$e_2=\{v_3,v_4,v_5\}$ and $e_3=\{v_1,v_4\}$. 
The number of partitions of $v$ is $B_5=52$ and the number of subsets of $E$ is only $2^3=8$.

\begin{figure}[!ht]
\begin{center}
\includegraphics[angle=0,width=4.5cm]{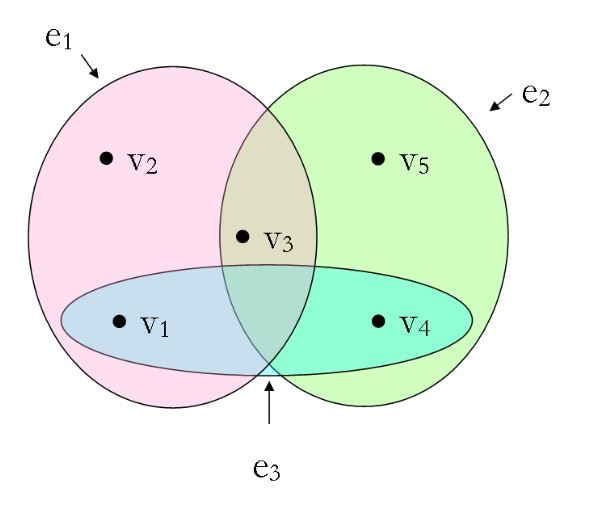}
\caption{Hypergraph with $n=5$ vertices and $m=3$ hyperedges.}
\label{fig1}
\end{center}
\end{figure}

\spc

In Table~\ref{tab1}, we enumerate all the subgraphs $H_i \in {\cal S}(H)$ by considering all 
subsets of $E$. In each case, we write the corresponding partition on $V$
induced by its connected components. We see for example that sub-hypergraphs $H_4$ and $H_7$ are in the same equivalence class, with
corresponding vertex partition: $\{ \{v_1, v_2,v_3, v_4, v_5\} \}$. 
The canonical representative of this class is the sub-hypergraph $H_7$.

\begin{table}[!ht]
\begin{center}
\begin{tabular}{c|c|c}
$i$ & $E_i \subseteq E$ & $p(H_i),~H_i=(V,E_i)$ \\
\hline
0 & $\emptyset$ & $\{ \{v_1\},\{v_2\},\{v_3\},\{v_4\},\{v_5\} \}$ \\
1 & $\{e_1\}$ & $\{ \{v_1, v_2, v_3\},\{v_4\},\{v_5\} \}$ \\
2 & $\{e_2\}$ & $\{ \{v_1\}, \{v_2\}, \{v_3,v_4,v_5\} \}$ \\
3 & $\{e_3\}$ & $\{ \{v_1, v_4\}, \{v_2\}, \{v_3\}, \{v_5\} \}$ \\
4 & $\{e_1,e_2\}$ & $\{ \{v_1, v_2, v_3, v_4, v_5\} \}$ \\
5 & $\{e_1,e_3\}$ & $\{ \{v_1, v_2, v_3, v_4\}, \{v_5\} \}$ \\
6 & $\{e_2,e_3\}$ & $\{ \{v_1, v_3, v_4, v_5\}, \{v_2\} \}$ \\
7 & $\{e_1,e_2,e_3\}$ & $\{ \{v_1, v_2,v_3, v_4, v_5\} \}$ \\
\end{tabular}
\end{center}
\caption{Enumerating the partitions induced by all sub-hypergraphs in ${\cal S}(H)$.}
\label{tab1}
\end{table}


\section{Estimating the modularity}
\label{section:algos}

In Section~\ref{section:search}, we showed that for hypergraph $H=(V,E)$,
the optimal partition of vertices in $V$ with respect to the strict modularity objective
function corresponds to the canonical representative of one 
of the equivalence classes of all sub-hypergraphs in ${\cal S}(H)$. 
This formulation provides a convenient way to describe algorithms to look for 
this optimal partition.

\spc

For very small hypergraphs $H=(V,E)$, we can enumerate all elements in ${\cal S}(H)$, find the corresponding vertex partitions and find the optimal solution with respect to the strict modularity function.
This is not feasible for most hypergraphs, and so we must rely on some heuristic search algorithms.
In this paper, we use two simple greedy algorithms, allowing us to work 
on hypergraphs with thousands of vertices. 
The development of good, efficient heuristic search algorithms over $q_H()$ is a topic for
further research.

\subsection{Greedy Random Algorithm}

We start with a very simple greedy random algorithm for hypergraph partitioning.
The details are given in Algorithm~\ref{alg:random} in the Appendix.
In a nutshell, we consider a random permutation of the edges.
For each edge in turn, we add it to the ``edge contribution'' factor if the overall
modularity improves.
We repeat this process for several permutations and keep the best result.

\subsection{Hypergraph CNM}

In the CNM algorithm for graph partitioning (Clauset-Newman-Moore; see, for example,~\cite{Newman} and~\cite{CNM}), we start with every vertex in its own part. 
At each step, we merge the two parts that yield the largest increase in modularity,
and we repeat until no such move exists.

\spc

We propose a simple version of the CNM algorithm for hypergraphs,
which we detail in Algorithm~\ref{algo:CNM_H} in the Appendix.
The idea in that algorithm is that in each step, we loop through every hyperedge not internal to a part, and we select the
one which, when it becomes internal to a newly merged part, yields the best modularity.




\section{Examples}
\label{section:ex}

\subsection{Synthetic Hypergraphs}

We generate hyperedges following the process described in \cite{Leordeanu2012}.
In a nutshell, we generate noisy points along 3 lines on the plane with different slopes, 30 points per line, to which we add 60 random points.
All sets of 3 or 4 points make up our hyperedges.
The hyperedges can be either all coming from the same line 
(which we call ``signal'') or not (which we call ``noise'').
We sample hyperedges for which the points are well aligned, 
and so that the expected proportion of signal vs.\ noise is 2:1.
We consider 3 different regimes: (i) mostly 3-edges, 
(ii) mostly 4-edges and (iii) balanced between 3 and 4-edges.
For the 3 regimes, we generate 100 hypergraphs and 
for each of the 300 hypergraphs, we apply the fast Louvain clustering
algorithm (see \cite{Blondel}) on
the weighted 2-section graph. In most cases, vertices coming from the same line are correctly put in the same part.

\spc 

In the left plot of  Figure \ref{fig:lines}, 
we plot the standard graph modularity vs.\ the Hcut value, 
which is simply the proportion of hyperedges that fall in two or more parts. 
The Louvain algorithm is not explicitly aiming at preserving the hyperedges, so we do not expect a high correlation between the two measures. 
In fact, fitting a regression line to the points from the balanced regime, we get a slope
of 0.0061 with $R^2$ value of 0.0008.

\begin{figure}[!ht]
\begin{center}
\includegraphics[angle=0,width=5.6cm]{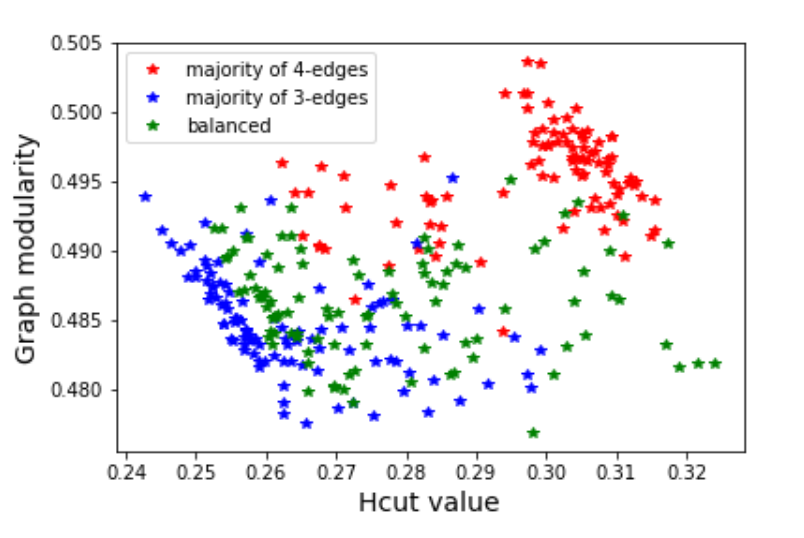}
\hspace{0.1cm}
\includegraphics[angle=0,width=5.4cm]{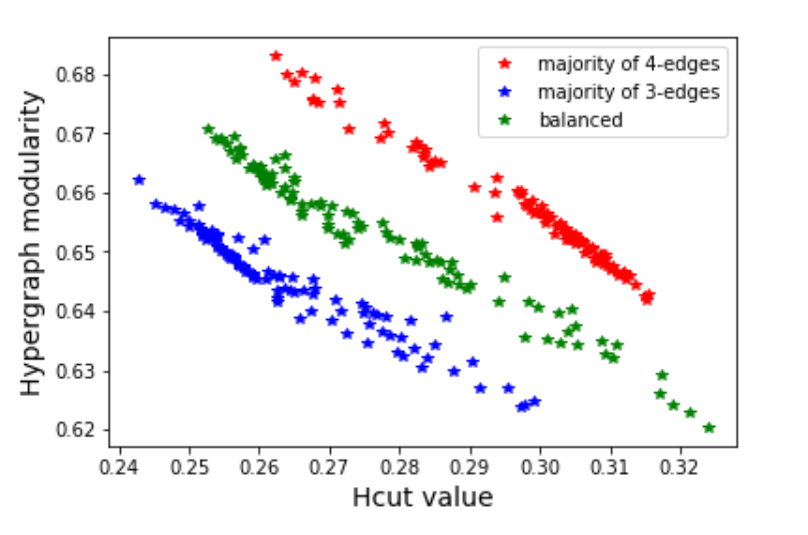}
\caption{Modularity vs Hcut: comparing graph and hypergraph modularity.}
\label{fig:lines}
\end{center}
\end{figure}

In the right plot Figure \ref{fig:lines}, we do the same, this time comparing our hypergraph
modularity with the Hcut values for the same partitions as in the left plot.
The correlation here is very high. For the balanced regime, linear regression yields
a slope of -0.6364 with $R^2$ value of 0.9693.
This is an illustration of the fact
that when we measure our proposed hypergraph modularity for different partitions, we 
are favouring keeping hyperedges in the same parts.







\subsection{DBLP Hypergraph}

The DBLP computer science bibliography database contains open bibliographic information on major computer science journals and proceedings. The DBLP database is operated jointly by University of Trier and Schloss Dagstuhl. The DBLP Paper data is available at \url{http://dblp.uni-trier.de/xml/}.

\spc

We consider a hypergraph of citations where each vertex represents an author and hyperedges are papers.
In order to properly match author names across papers we enhance the data with information scraped from journal web pages. The DBLP database contains the \url{doi.org} identifier. Since the same author names can be written differently we match author names of the paper across all three data sources, we use this information to obtain the journal name and retrieve paper author data directly from journal --- we update available author name data using ACM, ieexlpore, Springer and Elsevier/ScienceDirect databases. This can give good representation of author names for later matching.
For the analysis, we only kept the (single) large connected component.
We obtained a hypergraph with 1637 nodes,
865 edges of size 2, 470 of size 3, 152 or size 4 and 37 of size 5 to 7.


\spc

In Table \ref{table:dblp}, we show our results with the Louvain algorithm on the 2-section graph using modularity function $q_G()$,  
as well as the results with our Random and CNM algorithms on the hypergraph, using modularity function $q_H()$. We also ran our algorithms using the degree independent hypergraph modularity $q_H^{DI}()$, which gave us very similar results to using $q_H()$.
Note that Random algorithm is worse than CNM (which is to be expected).

\spc

Comparison of results of Louvain algorithm with CNM shows that there is a tradeoff between $q_H/q_H^{DI}$ and $q_G$ and importantly, Hcut value is lower for CNM algorithm.
The increased number of parts with our algorithms is mainly due to the presence of
singletons. 

\spc 

\begin{table}
\begin{center}
\begin{tabular}{c|c|c|c|c|c}
algorithm & $q_H()$ & $q_H^{DI}()$ & $q_G()$ & Hcut & \#parts \\
\hline
Louvain & 0.8613 & 0.8599 & 0.8805 & 0.1181 & 40 \\
Random    & 0.8485 & 0.8446 & 0.8198 & 0.0971 & 78 \\
CNM       & 0.8671 & 0.8655 & 0.8456 & 0.0945 & 92 \\
\end{tabular}
\caption{Results when partitioning the DBLP dataset with various algorithms. For Random algorithm we  presents the best result of 100 runs.}
\label{table:dblp}
\end{center}
\end{table}

Another observation is that the actual partitions obtained with 
objective function $q_G()$ (Louvain) and $q_H()$ (CNM, Random) are quite different. 
For the Louvain and CNM algorithms,
we found low values of 0.4355 for the ``adjusted RAND index'' and 0.4416 for the ``graph-aware adjusted RAND index'' (see \cite{Poulin2018}).
One of the difference lies in the number of edges of size 2, 3 and 4 that are cut
with the different algorithms, as we see in Table \ref{table:dblp-cut}. The algorithms based on $q_H()$
will tend to cut less of the larger edges, as compared to the Louvain
algorithm, at expense of cutting more size-2 edges.

\spc 

\begin{table}
\begin{center}
\begin{tabular}{c|c|c|c}
Algo & prop. of 2-edges cut & prop. of 3-edge cut & prop. of 4-edges cut \\
\hline
Louvain & 0.0382 & 0.1815 & 0.3158 \\
Random & 0.0509 & 0.1404 & 0.2039 \\
CNM & 0.0590 & 0.1277 & 0.1842 \\
\end{tabular}
\caption{Proportion of edges of size 2, 3 or 4 cut by various algorithms.}
\label{table:dblp-cut}
\end{center}
\end{table}





\section{Conclusion}
\label{section:conclusion}

In this paper, we presented a generalization of the Chung-Lu model for hypergraphs, which we used to define a modularity function on hypergraphs.
Interestingly, in hypergraph modularity case there is no one unique way to define modularity and we show that it depends on how strongly a user thinks that a hyperedges indicate members of the same community. If the belief is soft this leads to standard 2-section graph modularity. However, if it is strong a natural definition is strict hypergraph modularity, which we tested on numerical examples.

\spc

We have also showed that hypergraph modularity function can be simply ``specialized'' by considering  subsets of hyperedges independently. This can be useful, in particular, when a hypergraph consists of hyperedges representing different baseline hypergraphs, perhaps coming from different sources of data.

\spc

The objective of this paper was to develop a definition of hypergraph modularity. However, in order to show that this notion is numerically traceable, at least approximately, we provided the theoretical foundations for the development of algorithms using this modularity function that greatly reduce the solution search space.

\spc

A key natural question with any new measure is if it provides qualitatively different outcomes than existing ones. Therefore we have compared strict hypergraph modularity with a standard 2-section graph modularity. For this we have developed two simple heuristic algorithms. Using them we illustrated the fact that in comparison to 2-section graph modularity (optimized using Louvain algorithm) optimization using strict modularity function tends to cut a smaller number of hyperedges. Therefore the proposed measure is potentially highly valuable in application scenarios, where a hyperedge is a strong indicator that vertices it contains belong to the same community.

\spc

Hypergraph modularity is a new measure, and there is still a lot of work that should be done. First of all, the development of good, efficient heuristic algorithms would allow to look at larger hypergraphs. Such algorithms would allow us to perform a study over hypergraphs with different edge size distributions, comparing the hypergraph modularity function with other definitions such as graph modularity over the 2-section representation of the hyperedges, the degree-independent function $q_H^{DI}()$, and hypergraph modularity using the less strict majority rule. 


\section*{Acknowledgements}

The authors would like to thank Claude Gravel for useful discussions while 
developing the algorithms.

\clearpage

\section{Appendix}

\begin{figure}[!ht]
\begin{algorithm}[H]
\KwData{hypergraph $H=(V,E)$, number of steps $k$}
\KwResult{${\bf A}_{opt}$, a partition of $V$ with modularity $q_{opt}$}
Initialize $q_{opt}=-1$\;
$s \leftarrow 0$\;
\While{$s < k$}{
	Initialize ${\bf A}_{best} = \A$ with all $v \in V$ in its own part, $q_{best}$ is the corresponding modularity\; 
    Initialize $E^{'} = \emptyset$\; 
    Draw a random permutation $(e_1, \ldots, e_m)$ of $E$\; 
    $i \leftarrow 0$\;
    \While{$i < m$}{
      $i \leftarrow i+1$\;
      let $H^{'} = (V,E^{'} \cup \{ e_i \} )$\;
      find ${\bf A} = p(H')$ and compute $q_{dt}$ (degree tax)\;
      find $H^* = f([H^{'}])$ and compute $q_{ec} = |H^*|/m$ (edge contribution)\;
      \If{$q_{ec}-q_{dt} > q_{best}$}{
        $q_{best} = q_{ec}-q_{dt}$\;
        ${\bf A}_{best} = {\bf A}$\;
        $E^{'} = E^{'} \cup \{ e_i \}$\;
      }
    }
    \If{$q_{best} > q_{opt}$}{
    	$q_{opt} = q_{best}$\;
        ${\bf A}_{opt} = {\bf A}_{best}$\;
    }
    $s \leftarrow s + 1$\;
}
output: ${\bf A}_{opt}$ and $q_{opt}$\;
\caption{GreedyRandomPartition($H,k$) on a hypergraph $H$}
\label{alg:random}
\end{algorithm}
\end{figure}


\begin{figure}[!ht]
\begin{algorithm}[H]
\KwData{hypergraph $H=(V,E)$}
\KwResult{${\bf A}_{opt}$, a partition of $V$ with modularity $q_{opt}$}
Initialize $\A_{opt}$ the partition with all $v \in V$ in its own part, and $q_{opt}$ the corresponding  modularity\;
Initialize $E_0 = \emptyset$\;
\Repeat{$E_0 = E$}
{
  $q' = -1$\;
  \ForEach{$e \in E \setminus E_0$}{
    let $H'=(V,E_0 \cup \{e\})$\;
    find ${\bf A} = p(H')$ and compute $q_{dt}$ (degree tax)\;
    find $H^* = f([H^{'}]) = (V,E^*)$ and compute $q_{ec} = |H^*|/m$ (edge contribution)\;
    \If{$q_{ec}-q_{dt} > q'$}{
      $q' = q_{ec}-q_{dt}$\;
      ${\bf A}' = {\bf A}$\;
      $E' = E^*$\;
    }  
  }
  $E_0 = E'$\;
  \If{$q' \geq q_{opt}$}{
  	$q_{opt} = q'$\;
    ${\bf A}_{opt} = {\bf A}'$\;
  }
  [optional: break from loop if $q' < q_{opt}$]
}
output: ${\bf A}_{opt}$ and $q_{opt}$
\caption{SimpleCNM($H$) on a hypergraph $H$}
\label{algo:CNM_H}
\end{algorithm}
\end{figure}

\end{document}